\documentclass[10pt,a4paper,reqno]{amsart}

\usepackage{amsthm}
\usepackage{xcolor}
\usepackage{tikz}
\usepackage{tikz-cd}
\usepackage{standalone}
\usetikzlibrary{decorations.markings}
\usepackage[pagebackref]{hyperref}
\usepackage{comment}
\usepackage{cleveref}

\theoremstyle{plain}
\newtheorem{thm}{Theorem}

\newtheorem{cor}[thm]{Corollary}

\theoremstyle{definition}

\newtheorem{rem}[thm]{Remark}

\title{A note on the Lagrangian cobordism group of Weinstein sectors}
\author{Valentin Bosshard}

\begin{document}

\begin{abstract}
The aim of this note is to show that the Lagrangian cobordism group of a Weinstein sector is isomorphic to its middle-dimensional singular cohomology. As an application, a geometric description of Viterbo restriction for cobordism groups is obtained.
\end{abstract}

\maketitle

Let $(X,\mathfrak f)$ be a stopped Liouville manifold and denote $\Omega(X,\mathfrak f)$ its Lagrangian cobordism group as defined in \cite{BOS}. Recall that the generators of $\Omega(X,\mathfrak f)$ are exact Lagrangian submanifolds in $X$ that are conical and disjoint from the stop $\mathfrak f$ at infinity. We prove the following theorem:
\begin{thm}\label{thm1}
    There is a map
    \begin{equation*}
        i\colon\Omega(X,\mathfrak f) \to H_n(X,\partial_\infty X\setminus f)
    \end{equation*}
    taking a Lagrangian submanifold to its relative homology class.\footnote{We adopt the standard abuse of notation where $(X,\partial_\infty X\setminus f)$ denotes $(X_0,\partial X_0\setminus f)$ for a Liouville domain $X_0\subset X$ which completes to $X$. As $X_0$ is a deformation retract of $X$ the choice of $X_0$ is irrelevant.}
    Moreover, if $X$ is Weinstein and $\mathfrak f\subset \partial_\infty X$ a Weinstein hypersurface then the map $i$ is an isomorphism.
\end{thm}

In the presence of an explicit Weinstein presentation of $(X, \mathfrak f)$ the map $i$ can be described in terms of its core and cocores: If a Lagrangian submanifold intersects the relative core of $(X,\mathfrak f)$ transversely then it gets mapped to the sum of the cocores of this intersection. 

This point of view leads to the following description of Viterbo restriction for cobordism groups:
\begin{cor}\label{cor1}
    Let $Y_0\subset X$ be a Weinstein subdomain of a Liouville manifold $X$ and denote by $Y$ the Liouville completion of $Y_0$ in $X$. 
    Then there is a map
    $$\Omega(X)\to \Omega(Y)$$
    which sends a Lagrangian submanifold of $X$ to the sum of the cocores of the intersection of $L$ and the core of $Y$ whenever the intersection is transverse.
\end{cor}

The map $i$ can be naturally placed within a commutative diagram, alongside other invariants arising from Floer theory:
\begin{cor}\label{cor2}
    Let $X$ be a Liouville manifold and $\mathfrak f\subset \partial_\infty X$ be a Liouville hypersurface. Then the following diagram commutes
    \begin{center}
    \begin{tikzcd}
        \Omega(X,\mathfrak f) \arrow[r, "i"]\arrow[d,  "\Theta"'] &H_n(X,\partial_\infty X\setminus \mathfrak f) \arrow[r, "\mathcal A"] &SH^n(X, \mathfrak f)\\
        K_0(\mathcal W(X, \mathfrak f))\arrow[r, "\mathcal T"] & HH_0(\mathcal W(X, \mathfrak f))\arrow[ur, "\mathcal{OC}"']
    \end{tikzcd} 
    \end{center}
\end{cor}
Let us discuss the maps in Corollary \ref{cor2}. Denote by $\mathcal W(X,\mathfrak f)$ the partially wrapped Fukaya category of a stopped Liouville manifold $(X, \mathfrak f)$ as defined in \cite{GPS2}. The map $\Theta$, initially constructed in \cite{BC2} and transferred to the Liouville setting in \cite{BOS}, sends the cobordism class of a Lagrangian to its class in the Grothendieck group $K_0(\mathcal W(X,\mathfrak f))$ of the (derived) partially wrapped Fukaya category. It is well-defined as Lagrangian cobordisms produce iterated cone-decompositions in the partially wrapped Fukaya category, resulting in relations $K_0(\mathcal W(X,\mathfrak f))$. The map $\mathcal T$, known as Dennis trace or Chern character, is defined for any $A_\infty$-category from the Grothendieck group $K_0$ to Hochschild homology $HH_*$ in degree $*=0$. It sends the class of an object to its identity endomorphism, which represents a class in Hochschild homology.
The definition of symplectic cohomology $SH^*(X,\mathfrak f)$, the acceleration map $\mathcal A$ and the open-closed map $\mathcal OC$ were extended from Liouville manifolds to Liouville sectors in \cite{GPS1}.

Corollary \ref{cor2} generalizes a diagram that was previously introduced by Lazarev \cite{LAZ}. For a Weinstein manifold $X$ and a Weinstein hypersurface $\mathfrak f$, Lazarev constructs a well-defined map $\mathcal L\colon H_n(X,\partial_\infty X\setminus \mathfrak f)\to K_0(\mathcal W(X, \mathfrak f))$ without referring to the maps $i$ and only considers a related version of $\Theta$. The map $\mathcal L$ can be recovered by $\mathcal L=\Theta\circ i^{-1}$ as by Theorem \ref{thm1} the map $i$ is an isomorphism in Lazarev's setting. Lazarev's definition of $\mathcal L$ heavily depends on knowing the generators of $H_n(X,\partial_\infty X\setminus \mathfrak f)$ and $K_0(\mathcal W(X, \mathfrak f))$ in contrast to the definition of $i$ and $\Theta$ which can be defined in a more general setting.

\begin{rem} The map $\Theta$  is always surjective. As for injectivity, Lazarev argued in \cite{LAZ} that $H_n(X, \partial_\infty X) \to K_0(\mathcal W(X))$ is not injective when $X$ is a flexible Weinstein manifold, as the latter group vanishes. By Theorem \ref{thm1}, $\Omega(X)$ is isomorphic to $H_n(X, \partial_\infty X)$ and therefore the map $$\Omega(X) \overset{\Theta}\longrightarrow K_0(\mathcal W(X))$$
is also not injective. This appears to be the first known example where this phenomenon occurs, as far as the author is aware.
\end{rem}

\begin{rem}
The map $i$ already appears in \cite{BC3} when $X$ is closed, and more comparable to our setting, when $X$ a Liouville manifold, but in that paper only closed Lagrangian submanifolds and bounded cobordisms are allowed. Under such circumstances we have a map 
$$i\colon\Omega(X)\to H_n(X).$$
This corresponds to the case when the stop $\mathfrak f=\partial_\infty X$ is the entire boundary at infinity. In this setting the map $i$ is not always an isomorphism as demonstrated by the following example from \cite{HKK}:

Let $X$ be the cotangent bundle $T^*S^1$ and $\mathfrak f=\partial_\infty X$ be the boundary at infinity. Up to Hamiltonian isotopies, the only exact Lagrangian submanifold in $(X, \mathfrak f)$ is the zero-section. However, $X$ can be equipped with a non-trivial grading such that the zero-section is not part of the Fukaya category. Explicitly, the grading comes from a line field that is constant in the cotangent direction and rotates a line by $\pi$ when transported along $S^1$. Consequently, the Grothendieck group and the cobordism group vanish. On the other hand, the first homology group of $T^*S^1$ with coefficients twisted by the corresponding local system is non-trivial, it is $\mathbb Z_2$.
\end{rem}

\subsection*{Strategy of the proof of Theorem \ref{thm1}}
The map $i$ is well-defined as a Lagrangian cobordism $V\subset \mathbb C\times X$ can be regarded as a relative $(n+1)$-chain in $X$ when projected to $X$. Suppose now $X$ is Weinstein and the stop $\mathfrak f$ is a Weinstein hypersurface. Then not only the relative homology but also the Lagrangian cobordism group is generated by cocores (see \cite{HH,HHL}). Furthermore, \cite{LAZ} showed that the relations defining relative homology can be realized in well-chosen coordinates by a boundary connected sum of certain cocores which is also realized by a Lagrangian cobordism.

\subsection*{Acknowledgements} 
This note is a result of the author's visit to the University of Edinburgh. The author greatly thanks Jeff Hicks for helpful discussions, particularly bringing to his attention that one can look at a cobordism as a chain as already done by Biran and Cornea, at least in the closed setting. Many thanks also to my advisor, Paul Biran, for various remarks improving the exposition and readability of this note. The author was partially supported by the Swiss National Science Foundation (grant number  $200021\_204107$).

\section{Definitions}
In this note, we assume that all Lagrangian submanifolds are oriented, exact with compactly supported primitive, and disjoint from the stop $\mathfrak f$ at infinity. One might further want to restrict to Lagrangian submanifolds that can be equipped with a spin structure. If the symplectic manifold comes with a grading one might focus only on graded Lagrangian submanifolds. A Lagrangian equipped with the above additional structures is referred to as a \textbf{Lagrangian brane}. These additional structures are necessary for the diagram in Corollary \ref{cor2} (see \cite{LAZ}). When using gradings, homology needs to be twisted by the corresponding local system.

\subsection{Weinstein, Liouville and stops}
The following symplectic manifolds will appear in this note.
$$
\begin{array}{@{}c@{\;}c@{\;}c@{\;}c@{\;}c@{\;}c@{\;}c@{}}
\text{Weinstein manifolds}&\subset &\text{Weinstein pairs}&\cong &\text{Weinstein sectors} & \subset & \text{weakly Weinstein manifolds} \\
\rotatebox[origin=c]{-90}{$\subset$} &&\rotatebox[origin=c]{-90}{$\subset$}&&\rotatebox[origin=c]{-90}{$\subset$}&&\rotatebox[origin=c]{-90}{$\subset$}\\
\text{Liouville manifolds}&\subset &\text{Liouville pairs}&\cong &\text{Liouville sectors} &\subset &\text{stopped Liouville manifolds}
\end{array}
$$

References providing extended discussion are \cite{ELI2, CHA2} for Weinstein manifolds, pairs and sectors, and \cite{GPS2} for the other notions. We summarize below the definitions for the reader's convenience.

\vspace{5mm}

A \textbf{Liouville domain} is a symplectic manifold $X_0$ manifold of dimension $2n$ with boundary $\partial X_0$ equipped with an exact symplectic form $\omega=d\lambda$ where the $\omega$-dual vector field $Z$ associated to the one-form $\lambda$ is outward pointing along $\partial X_0$. A \textbf{Liouville manifold} $X$ is the completion of a Liouville domain $X_0$ by attaching the positive half of the symplectisation of the contact manifold $\partial X_0$ to $X_0$. 

A \textbf{stop} is a closed subset $\mathfrak f\subset \partial X_0$, and  $(X_0,\mathfrak f)$ is called \textbf{stopped Liouville domain}. Also completing the stop by the Liouville flow in positive time produces a \textbf{stopped Liouville manifold} $(X, \mathfrak f)$. The data of a stopped Liouville manifold does not come with a choice of a stopped Liouville domain. It only assumes its existence, namely the stop $\mathfrak f$ is assumed to live in the ideal boundary, the Liouville boundary at infinity $\partial_\infty X$.

A \textbf{Liouville pair} $(X,\mathfrak f)$ is a Liouville manifold $X$ and a stop $\mathfrak f$, where $\mathfrak f\subset \partial_\infty X$ is a $(2n-2)$-dimensonal Liouville domain. The result of what we get after removing an open neighbourhood of $\mathfrak f$ is called a \textbf{Liouville sector}, which can be viewed as a specific Liouville manifold with boundary $\mathfrak f\times \mathbb R$).

The \textbf{relative core}  $\mathrm{core}(X, \mathfrak f)$ of a stopped Liouville manifold $(X, \mathfrak f)$ is the set of points in $X$ that, under the Liouville flow associated to the vector field $Z$, neither escape any Liouville domain nor end up in the stop $\mathfrak f$ eventually. If $\mathrm{core}(X, \mathfrak f)$ is \textbf{mostly Lagrangian}, meaning that after deleting a closed subset contained in the image of an at most $(n-1)$-dimensional manifold, it becomes Lagrangian in $X$, and any connected component of the Lagrangian part of the relative core admits a transverse Lagrangian intersecting that component exactly once (these Lagrangian submanifolds are called \textbf{generalized cocores}), then $(X,\mathfrak f)$ is called \textbf{weakly Weinstein}.

A \textbf{Weinstein manifold} $X$ is a Liouville manifold that admits a Morse function for which the Liouville vector field $Z$ is gradient-like. Weinstein manifolds can be reconstructed iteratively using copies of the Liouville manifolds $H^k=T^*B^{k}\times B^{2n-k}$, where $k=0,\ldots, n$. A \textbf{Weinstein pair} $(X,\mathfrak f)$ is a Weinstein manifold $X$ and a stop $\mathfrak f$ which is a $(2n-2)$-dimensional Weinstein domain. The relative core of a Weinstein pair is mostly Lagrangian with cocores of the smooth components being cotangent fibers of critical dimension handles $H^n$. Removing an open neighbourhood of $\mathfrak f$, as in the Liouville case, produces a \textbf{Weinstein sector}.

\subsection{Lagrangian cobordisms}
We use the definition of a Lagrangian cobordisms in stopped Liouville manifolds given in \cite{BOS}.

Let $(X,\mathfrak f)$ be a stopped Liouville manifold. Denote by $(\mathbb C,\mathfrak f_m)$ the completion of the disk with stop $\mathfrak f_m$ consisting of $(m+1)$ points on the boundary.

A \textbf{Lagrangian cobordism} is a Lagrangian submanifold $V$ in the product of stopped Liouville manifolds $(\mathbb C, \mathfrak f_m)\times (X, \mathfrak f)$ satisfying the following conditions:
\begin{itemize}
    \item $V$ is invariant under the Liouville flow outside a compact set $K\times X_0\subset \mathbb C\times X$,
    \item $V$ is a conicalization\footnote{For this note, it is enough to think of an honest product as one can conicalize and deconicalize the products that appear as ends, see \cite{GPS2}.} $\gamma_j\tilde \times L_j$ of products $\gamma_j\times L_j$ on $(\mathbb C\setminus K)\times X_0$, where $\gamma_j$ are radial rays at infinity in $\mathbb C$ passing through the arc between the stop $j$ and the stop $j+1$ on $\partial K$ and $L_j$ are Lagrangian submanifolds in $(X, \mathfrak  f)$ for $j=0,\dots, m$. The Lagrangian submanifolds $\gamma_j\tilde \times L_j$ are referred to as the \textbf{ends of the cobordism}.
\end{itemize}
\begin{center}
\begin{figure}
    \begin{tikzpicture}[scale=0.7]
\draw[fill=green, draw=green]    (-3.5,0)-- (-2,0) .. controls (0,0) and (0,0) .. (-1,2)--(-1.7,3.3)--(-1,2).. controls (0,0) and (0,0.8) .. (1.5,1)--(4,2)--(1.5,1) .. controls (1,0.7) and (1,-0.4) ..(2,-0.4)--(4,-0.7)--(2,-0.4).. controls (0,-2) and (0,0) ..(-1,0);

\node[green] at (-3.5,0.5) {$\gamma_2\tilde \times L_2$};
\node[green] at (3.8,1.3) {$\gamma_0\tilde  \times L_0$};
\node[green] at (3.8,-1.3) {$\gamma_3\tilde \times L_3$};
\node[green] at (-0.5,3.5) {$\gamma_1\tilde \times L_1$};
\node at (-2.4,-2.3) {$(\partial K, \mathfrak f_3)$};

\draw[dotted] (0,0) ellipse (2.5 and 2.5);
\node[shape = circle,fill = black, inner sep=2*0.7pt] at (90:2.5) {};
\node[shape = circle,fill = black, inner sep=2*0.7pt] at (160:2.5) {};
\node[shape = circle,fill = black, inner sep=2*0.7pt] at (5:2.5) {};
\node[shape = circle,fill = black, inner sep=2*0.7pt] at (-90:2.5) {};

\end{tikzpicture}
    \caption{A Lagrangian cobordism $V$ can be visualized by its projection to $\mathbb C$. As $V$ is conical outside a compact set no information is lost when only considering $V\cap \left (K\times X_0\right)$. In green, the projection of $V\cap \left (\mathbb C\times X_0\right)$ to $\mathbb C$ is drawn. The Lagrangian cobordism $V$ has 4 ends $L_0,\ldots, L_3$ parametrized by the rays $\gamma_0, \ldots, \gamma_3$ in $\mathbb C$ in this example. The dotted circle $\partial K$ is the boundary at infinity of $\mathbb C$ stopped at 4 points $\mathfrak f_3\subset \partial K$.}
\label{fig}
\end{figure}
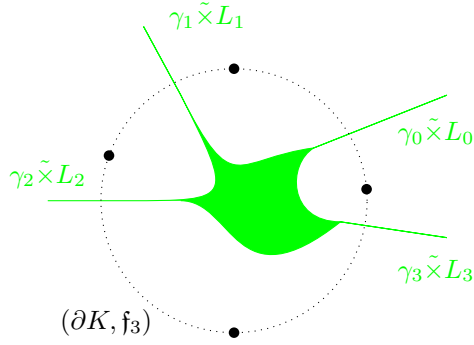
\end{center}

The \textbf{Lagrangian cobordism group} $\Omega (X, \mathfrak f)$ is the free abelian group generated by all Lagrangian submanifolds in $(X, \mathfrak f)$ modulo the relations $L_{0}+\cdots +L_m=0$ whenever there is a Lagrangian cobordism with ends $L_0\ldots, L_{m}$.

\section{Proofs}
\begin{proof}[Proof of Theorem \ref{thm1}] After deconicalizing the ends of the Lagrangian cobordism $V$ the submanifold $\widetilde V=V\cap (K\times X_0)$ has boundary $$\partial \widetilde V =\left(\bigcup_{i=0}^n (\gamma_j\cap \partial K)\times L_j\right) \cup \left(V\cap (K\times \partial X_0)\right)\subset \partial K\times X_0 \cup K\times \partial X_0.$$ Let $i_{\widetilde V}:\widetilde V\to K\times X_0$ be the inclusion map. By triangulating the submanifold $\widetilde V$ such that it restricts to a triangulation of its boundary $\partial \widetilde V$, $\widetilde V$ gives rise to a singular $(n+1)$-chain $\sigma \in S_{n+1}(K\times X_0$). Here and in what follows we denote by $S_k(-)$ the group of singular chains. Let $\pi_X:\mathbb C\times X \to X$ be the projection. Then $\pi_X\circ \sigma\in S_{n+1}(X_0)$ is a singular $(n+1)$-chain whose boundary is a triangulation of $\bigcup_{i=0}^n \{p_j\}\times L_j$ and the (possibly not embedded) set $\pi_X(K\times \partial X_0)\subset (\partial X_0\setminus \mathfrak f)$. Considering $\pi_X\circ \sigma$ as a relative chain, its boundary is equal to $\sum_{i=0}^n L_j + S_{n}(\partial X_0\setminus \mathfrak f)\in S_n(X_0,\partial X_0\setminus \mathfrak f)$. This shows that $\left[\sum_{i=0}^n L_j\right ] = 0$ in $H_n(X_0,\partial X_0\setminus \mathfrak f)$ whenever there is a Lagrangian cobordism with ends $L_0,\ldots, L_m$. 

Suppose now $(X, \mathfrak f)$ is weakly Weinstein. Then \cite{HH,HHL} prove that $\Omega(X,\mathfrak f)$ is generated by generalized cocores. Their argument is the following: let $L$ be a Lagrangian submanifold in $(X,\mathfrak f)$. We can assume that $L$ intersects the relative core $\mathrm{core}(X,\mathfrak f)$ transversely in the smooth part as this can be achieved by a compactly supported Hamiltonian diffeomorphism, which does not change the class of the Lagrangian in the Lagrangian cobordism group. Next, applying the Liouville flow to $L$, produces a Lagrangian suspension cobordism which can be cut off outside a compact set of $X$ and yields a Lagrangian cobordism with left end $L$ and right ends consisting of Lagrangian submanifolds that are Hamiltonian isotopic to generalized cocores of $L\cap \mathrm{core}(X,f)$.

Let us further restrict our setting to a Weinstein pair $(X, \mathfrak f)$. According to \cite{ELI2, CHA2} there is a Morse function $h$ compatible with both the Liouville structure on $X$ and $\mathfrak f$ whose cellular decomposition is a decomposition into standard handles $H^k$. Let us describe the homology $H_n(X, \partial_\infty X\setminus \mathfrak f)$ in terms of relative Morse homology. By Poincar\'e duality this is equivalent to describe relative Morse cohomology $H^n(X, \mathfrak f)$. Since $h$ has no critical points of index $n+1$, $H^n(X, \mathfrak f)$ is the free abelian group generated by critical points of $h$ of index $n$ modulo relations coming from critical points of index $n-1$. That is, for any critical point $x$ of index $(n-1)$ consider all positive gradient flow lines of $h$ that start at $x$ and end in a critical point $y$ of index $n$. Then the sum of all such critical points (counted with orientation and multiplicity) vanishes. Cocores are the unstable manifolds of the critical points of index $n$ and are Lagrangian submanifolds. This readily shows that the map $i$ is surjective. 

To prove injectivity of the map $i$, we are left to show that the relations in homology described above are realized by Lagrangian cobordisms. This is done in \cite{LAZ}, and we only outline the main argument for the reader’s convenience. In the following, we describe the relations to define Morse homology in degree $n$ that were outlined above but now in terms of handles. Let $H_x = T^*B^{n-1}\times  B^2$ be the $(n-1)$-handle associated to a critical point $x$ of index $n-1$ and $H_y = T^*B^n$ be the $n$-handle associated to a critical point $y$ of index $n$. The gradient flow line between $x$ and $y$ is the unstable manifold of $x$ intersected with the stable manifold of $y$. Lazarev \cite{LAZ} shows that there is a coordinate transformation (see Figure \ref{fig2}), such that all positive flow lines in $H_x = T^*B^{n-1}\times B^2$ leaving $x$ and converging to critical points $y_0,\ldots, y_m$ of index $n$ are constant in the first coordinate and radial in $B^2$. Denote by $z_0,\ldots, z_m$ the intersection of the flow lines with $\partial B^2$. Treating  $z_j$ as a connected component of a stop in $\partial B^2$ it admits a linking disc $D_j\subset B^2$. The Lagrangian disks $B^{n-1}\times D_j \subset T^*B^{n-1}\times \partial B^2$ are Lagrangian isotopic to the cocores associated to the critical points $y_j$, respectively. So the relation between the cocores in $H^n(X, \mathfrak f)$ coming from $H_x$ is that the sum of the cocores associated to $y_0,\ldots, y_m$ vanishes. This relation can also be realized as the boundary connected sum of the linking disks $D_0,\ldots, D_m$, which itself is realized as an iterated surgery Lagrangian cobordism as described in \cite{BOS} and hence is a relation that defines $\Omega(X,\mathfrak f)$.
\end{proof}

\begin{center}
\begin{figure}
    \begin{tikzpicture}[scale=0.7]

\draw (0,0) -- (0:2.5);
\draw (0,0) -- (50:2.5);
\draw (0,0) -- (120:2.5);
\draw (0,0) -- (200:2.5);
\draw (0,0) -- (270:2.5);

\draw[blue]    (20:2.5) .. controls (0:2) and (0:2) .. (-20:2.5);
\draw[blue]    (70:2.5) .. controls (50:2) and (50:2) .. (30:2.5);
\draw[blue]    (140:2.5) .. controls (120:2) and (120:2) .. (100:2.5);
\draw[blue]    (220:2.5) .. controls (200:2) and (200:2) .. (180:2.5);
\draw[blue]    (290:2.5) .. controls (270:2) and (270:2) .. (250:2.5);

\node[blue] at (-20:3) {$D_0$};
\node[blue] at (30:3) {$D_1$};
\node[blue] at (100:3) {$D_2$};
\node[blue] at (180:3) {$D_3$};
\node[blue] at (250:3) {$D_4$};

\node at (0:3) {$z_0$};
\node at (50:3) {$z_1$};
\node at (120:3) {$z_2$};
\node at (200:3) {$z_3$};
\node at (270:3) {$z_4$};

\node[red] at (320:0.5) {$x$};
\node[shape = circle,fill = red, inner sep=2*0.7pt] at (0:0) {};

\draw[dotted] (0,0) ellipse (2.5 and 2.5);
\node[shape = circle,fill = black, inner sep=2*0.7pt] at (0:2.5) {};
\node[shape = circle,fill = black, inner sep=2*0.7pt] at (50:2.5) {};
\node[shape = circle,fill = black, inner sep=2*0.7pt] at (120:2.5) {};
\node[shape = circle,fill = black, inner sep=2*0.7pt] at (200:2.5) {};
\node[shape = circle,fill = black, inner sep=2*0.7pt] at (270:2.5) {};
\end{tikzpicture}
    \caption{An $(n-1)$-handle $T^*B^{n-1}\times B^2$ associated to a critical point $x$ projected to $B^2$. In well-chosen coordinates, positive gradient flow lines from the critical point $x$ of index $(n-1)$ to critical points of index $n$ are constant in the first coordinate and point radially outward. Each such flow line gives rise to a linking disk $D_j$ orthogonal to the flow line.}
\label{fig2}
\end{figure}
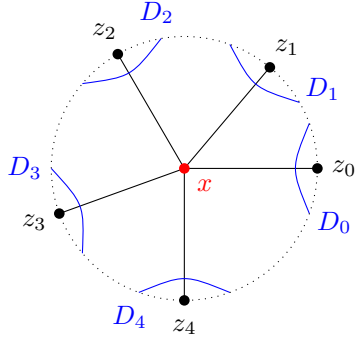
\end{center}

\begin{proof}[Proof of Corollary \ref{cor1}]
    By Poincar\'e duality, we have $H_n(X, \partial_\infty X)\cong H^n(X)$ and $H_n(Y, \partial_\infty Y)\cong H^n(Y)$. Viterbo restriction for cobordism groups is then given by postcomposing $i$ with restriction in cohomology $H^n(Y)\to H^n(X)$ induced by the inclusion $X\to Y$, and eventually using again $i$ for the Weinstein manifold $Y$.
\end{proof}

\begin{proof}[Proof of Corollary \ref{cor2}]
The statement follows from Proposition 5.13 in \cite{GPS1}, which states that the following diagram commutes for each Lagrangian $L$ in $(X,\mathfrak f)$:
  \begin{center}
  \begin{tikzcd}
        H^*(L)\arrow[r]\arrow[rd]& H^{*+n}(X,\mathfrak f) \arrow[r, "\mathcal A"] &SH^{*+n}(X, \mathfrak f).\\
        & HW^*(L,L)\arrow[ur, "\mathcal{OC}"']
    \end{tikzcd} 
  \end{center}
  Commutativity of the diagram in Corollary \ref{cor2} is then true by the following.
  The composition $\mathcal T\circ \Theta:\Omega(X,\mathfrak f)\to HH_0(\mathcal W(X, \mathfrak f)$ takes the class of $L$ in $\Omega(X,\mathfrak f)$ to the unit endomorphism $e_L\in CW^0(L,L)$. By the diagram above, applying the open-closed map $\mathcal{OC}$ to $e_L$ is the same as applying the acceleration map $\mathcal A$ to the image of the fundamental class $[L]$ in $H^0(L)\to H^{*+n}(X,\mathfrak f)\cong H_n(X,\partial_\infty X\setminus \mathfrak f).$ This is the same as applying the first row in the diagram in Corollary \ref{cor2} to the class of $L$ in $\Omega(X,\mathfrak f)$.
  
\end{proof}

\bibliography{bibi}
\bibliographystyle{aomalpha}
\end{document}